\documentclass{amsart}

\usepackage{graphicx}
\usepackage{proof}
\usepackage{tikz}
\usepackage{caption,setspace}
\newtheorem{theorem}{Theorem}[section]
\newtheorem{lemma}[theorem]{Lemma}
\newtheorem{proposition}{Proposition}
\newtheorem{corollary}{Corollary}
\theoremstyle{definition}
\newtheorem{definition}[theorem]{Definition}

\theoremstyle{remark}

\numberwithin{equation}{section}

\renewcommand{\phi}{{\varphi}}

\newcommand{\NS}{\mathcal{S}}
\newcommand{\NA}{\mathbf{A}}
\newcommand{\V}{\mathcal{V}}
\renewcommand{\neg}{{\sim}}

\newcommand\Tstrut{\rule{0pt}{2.6ex}}

\begin{document}

\title{Algebraic semantics for Nelson's logic~$\NS$}

\author{Thiago Nascimento}
\address{Programa de P\'{o}s-gradua\c{c}\~{a}o em Sistemas e Computa\c{c}\~{a}o, UFRN, Natal, Brazil}
\email{thiagnascsilva@gmail.com}

\author{Umberto Rivieccio}
\address{Departamento de Inform\'atica e Matem\'atica Aplicada, UFRN, Natal, Brazil}
\email{urivieccio@dimap.ufrn.br}

\author{Jo\~{a}o Marcos}
\address{Departamento de Inform\'atica e Matem\'atica Aplicada, UFRN, Natal, Brazil}
\email{jmarcos@dimap.ufrn.br}

\author{Matthew Spinks}
\address{Universit\`{a}  degli Studi di Cagliari, Cagliari, Italy}
\email{mspinksau@yahoo.com.au}





\begin{abstract}
Besides the better-known Nelson's Logic and Paraconsistent Nelson's Logic, in ``Negation and separation of concepts in constructive systems'' (1959), David Nelson introduced a logic called~$\NS$ with the aim of analyzing the constructive content of provable negation statements in mathematics. 
Motivated by results from Kleene, in ``On the Interpretation of Intuitionistic Number Theory'' (1945), Nelson investigated a more symmetric recursive definition of truth, according to which a formula could be either primitively verified or refuted. 
The logic~$\NS$ was defined by means of a calculus lacking the contraction rule and having infinitely many schematic rules, and no semantics was provided. 
This system received little attention from researchers; it even remained unnoticed that on its original presentation it was inconsistent. Fortunately, the inconsistency was caused by typos and by a rule whose hypothesis and conclusion were swapped. 
We investigate 
a corrected version of the logic~$\NS$, and focus on its propositional fragment, showing that it is algebraizable in the sense of Blok and Pigozzi (in fact, implicative) with respect to a certain 
class of involutive residuated lattices. 
We thus introduce the first (algebraic) semantics for~$\NS$ as well as a finite Hilbert-style calculus equivalent to Nelson's presentation; 
we also compare~$\NS$ with the other two above-mentioned logics of the Nelson family.
Our approach is along the same lines of (and partly relies on) previous algebraic work on Nelson's logics  
due to M.~Busaniche, R.~Cignoli, S.~Odintsov, M.~Spinks and R.~Veroff.
\end{abstract}

\maketitle

\section*{Introduction}
\noindent
To study the notion of constructible falsity, David Nelson introduced a number 
of systems of non-classical logic that
combine an intuitionistic approach to truth with a dual-intuitionistic treatment of falsity. 
Nelson's logics ($\NS$, $\mathcal{N}3$, and $\mathcal{N}4$) accept some notable theorems of classical logic, such as $\neg\neg\phi\Leftrightarrow\phi$, while rejecting others, such as $(\phi \Rightarrow (\phi \Rightarrow \psi)) \Rightarrow (\phi \Rightarrow \psi)$ and $(\phi \land \neg \phi) \Rightarrow \psi$. Nelson~introduced these logics with the aim of studying constructive proofs in Number Theory. To such an end, he gave a definition of truth \cite[Definition 1]{Nel49} (analogous to Kleene's \cite[p. 112]{Kl45}) 
according to which either a formula or its negation should be realized by some natural number.

Nelson's logic $\mathcal{N}3$ was introduced in \cite{Nel49} and $\mathcal{N}4$, a paraconsistent version of~$\mathcal{N}3$,  was introduced in \cite{AlNe84}. 
$\mathcal{N}3$ is in fact an axiomatic extension of $\mathcal{N}4$ by 
the axiom\footnote{The presence of two implications, the \emph{strong} one  ($\Rightarrow$) mentioned earlier and the \emph{weak} one ($\to$), is a distinctive feature
	of Nelson's logics; more on this below.} $\neg \phi \to (\phi \to \psi)$. The logic $\mathcal{N}3$ is by now well studied, both via a proof-theoretic approach and through algebraic methods; in particular, Odintsov \cite{Od08} proved that $\mathcal{N}4$ (thus also $\mathcal{N}3$) is algebraizable \`a la Blok-Pigozzi \cite{BP89}.



In~\cite{Nel59} Nelson also introduced the logic~$\NS$, aimed at the study of realizability. 
As  suggested by L.~Humberstone~\cite[Ch.\ 8.2, p.\ 1239--40]{Hu}, the introduction of~$\NS$ can perhaps also be viewed as
an attempt to remedy what some logicians consider an undesirable feature of  $\mathcal{N}3$ (and $\mathcal{N}4$), 
namely the fact that there are formulas $\phi$, $\psi$ in $\mathcal{N}3$ (and $\mathcal{N}4$) that are mutually interderivable  but
such that their negations   $\neg \phi$, $\neg \psi$ fail to be interderivable.
 It is useful to recall that these two phenomena are in general disassociated;
the latter stems from the failure of the contraposition law for the so-called \emph{weak implication} 
connective $\to$ of $\mathcal{N}3$ (and $\mathcal{N}4$), while the former entails that 
$\mathcal{N}3$ and $\mathcal{N}4$ are non-congruential  (or, as other authors say,
non-self-extensional) logics: that is, the logical interderivability relation fails to be a congruence of 
the formula algebra. Now, while $\NS$ is also a non-congruential logic, its implication connective
(here denoted $\Rightarrow$) does satisfy the contraposition law: in fact in~$\NS$ one has that   
$(\phi \Rightarrow \psi) \land ( \psi \Rightarrow \phi)$ is a theorem  if and only if 
$(\neg \phi \Rightarrow \neg \psi) \land (  \neg \psi \Rightarrow \neg \phi)$ is a theorem. 
In other words $\NS$, although non-congruential if we look at its interderivability relation, 
enjoys at least 
\emph{$\Leftrightarrow$-congruentiality}
in
Humberstone's terminology (relative to the bi-implication $\Leftrightarrow$ 
defined, in the usual way, as follows: $ \phi \Leftrightarrow \psi : =$ $( \phi \Rightarrow  \psi) \land (   \psi \Rightarrow  \phi)$)\footnote{Actually, we now know that $\mathcal{N}3$ (and $\mathcal{N}4$) are also $\Leftrightarrow$-congruential for a suitable choice
	of implication $\Rightarrow$ (called \emph{strong implication}) that can be defined using the weak one $\to$; but  Nelson may well not have been aware of this while writing~\cite{Nel59}.}. 

Nelson's original presentation of~$\NS$ 
has infinitely many schematic rules and no algebraic semantics;
\cite{Nel59} also leaves unclear 
whether $\mathcal{N}3$ is comparable with~$\NS$ (and if so, which of the two is stronger). 
Unlike its relatives $\mathcal{N}3$ and $\mathcal{N}4$, the logic~$\NS$ received little attention after~\cite{Nel59} and 
basic questions about it 
were left open, for example: 
Is $\NS$ algebraizable? Can~$\NS$ be finitely axiomatized? What are the exact relations between~$\NS$ and $\mathcal{N}3$, and between~$\NS$ and $\mathcal{N}4$? 
In the present paper we will use the modern techniques of algebraic logic to answer these questions.  

Our study will follow the same lines of previous papers by M.~Busaniche, R.~Cignoli, S.~Odintsov, M.~Spinks and R.~Veroff
on (algebraic models of) $\mathcal{N}3$ and $\mathcal{N}4$ (see, e.g., 	\cite{Busa10,Ci86,Od08,SpVe08a,SpinXX}), which in turn
rely on classic work by H.~Rasiowa on the algebraization of non-classical logics.  
These investigations have shown that the algebraic approach to Nelson's logics may be particularly insightful, as it allows
to view them as either  conservative
expansions of the negation-free fragment of intuitionistic logic 
by the addition of a new unary logical connective of \emph{strong negation} ($\neg$) or as
axiomatic extensions of well-known substructural/relevance logics.
The first perspective allows us to establish a particularly useful link between
algebraic models of $\mathcal{N}3$/$\mathcal{N}4$ and models of intuitionistic logic
(via  the so-called \emph{twist-structure} construction --- see especially \cite{Od08}), while the second
affords the possibility of exploiting general results and techniques that have been introduced 
in the study of residuated structures; this is the approach of 
\cite{SpVe08a,SpinXX} as well as \cite{Igpl}, and that we shall also take in the present paper  (see especially Subsection~\ref{ss:alt}).

The paper is organized as follows. In Section \ref{section2} we present the propositional fragment of the logic~$\NS$ and highlight some of its theorems, which will later be used to establish its algebraizability. In Section~\ref{section3} we prove that~$\NS$ is algebraizable and present 
its equivalent algebraic semantics. In Section~\ref{section4} we provide another calculus for~$\NS$, one that has a finite number of schematic axioms and only one schematic rule (\textit{modus ponens}). We point out that having only one rule makes it 
easy
to prove the Deduction Metatheorem in the standard way using induction over derivations.
In Section~\ref{section5} we prove that $\mathcal{N}3$ is a proper axiomatic extension of~$\NS$, and that~$\NS$ and $\mathcal{N}4$ are not extensions of each other.
Proofs of some of the main new results are to be found in an Appendix to this paper.

\section{Nelson's logic~$\NS$}\label{section2}

In this section we recall Nelson's original presentation of the propositional fragment of~$\NS$ \cite{Nel59} and we highlight some theorems of~$\NS$ that will be used further on to establish its algebraizability.

As is now usual, here we take a sentential logic~$\mathcal{L}$ to be a structure containing a subs\-titu\-tion-invariant consequence relation~$\vdash_\mathcal{L}$ defined over an algebra of formulas~$\mathbf{Fm}$ freely generated by a denumerable set of propositional variables $\{p,q,r,\ldots\}$ over a given language~$\Sigma$.  We will henceforth refer to algebras using boldface strings (such as $\mathbf{Fm}$ and $\mathbf{A}$), and use the corresponding italicized version of these same strings (such as $Fm$ and $A$) to refer to their corresponding carriers. Fixing a given logic, we will use $\varphi$, $\psi$ and~$\gamma$, possibly decorated with subscripts, to refer to arbitrary formulas of it.

\begin{definition}\label{ax:S}
	Nelson's logic $\NS = \langle \mathbf{Fm}, \vdash_{\NS} \rangle$ is 
	the sentential logic in the language $\langle \land, \lor, \Rightarrow, \neg, \bot \rangle $ of type $\langle 2, 2, 2, 1, 0 \rangle$
	defined by the Hilbert-style calculus with the 
	schematic axioms and rules listed below. 
	As usual, $\phi \Leftrightarrow \psi$ will be used to abbreviate $(\phi \Rightarrow \psi) \land (\psi \Rightarrow \phi )$.
\end{definition}

\noindent
{\bf Axioms}
{\small
	\begin{description}
		\item[\texttt{(A1)}] $\phi \Rightarrow \phi$ 
		\item[\texttt{(A2)}] $\bot \Rightarrow \phi$
		\item[\texttt{(A3)}] $\neg \phi \Rightarrow (\phi \Rightarrow \bot)$
		\item[\texttt{(A4)}] $\neg \bot$
		\item[\texttt{(A5)}] $(\phi \Rightarrow \psi) \Leftrightarrow (\neg \psi \Rightarrow \neg \phi)$
	\end{description}
}

\noindent
{\bf Rules}\medskip

{\small
	\noindent
	\begin{tabular}{lll}
		\infer[\texttt{(P)}]{\Gamma \Rightarrow (\psi \Rightarrow (\phi \Rightarrow \gamma))}{\Gamma \Rightarrow (\phi \Rightarrow (\psi \Rightarrow \gamma))} 
		\Tstrut\ \ 
		&  
		\infer[\texttt{(C)}]{\phi \Rightarrow (\phi \Rightarrow \gamma)}{\phi \Rightarrow (\phi \Rightarrow (\phi \Rightarrow \gamma))}
		\Tstrut\ \  
		&
		\infer[\texttt{(E)}]{\Gamma \Rightarrow \gamma }{\Gamma \Rightarrow \phi  \quad \phi \Rightarrow \gamma} 
		\Tstrut\ \ 
		\\[3mm] 
		\infer[\mathrm{(\Rightarrow \texttt{l})}]
		{\Gamma \Rightarrow ((\phi \Rightarrow \psi) \Rightarrow \gamma)}{\Gamma \Rightarrow \phi  \quad \psi \Rightarrow \gamma} &   \infer[\mathrm{(\Rightarrow \texttt{r})}]{\phi \Rightarrow \gamma}{\gamma}  &
		\infer[\mathrm{(\land \texttt{l1})}]{(\phi \land \psi) \Rightarrow \gamma}{\phi \Rightarrow \gamma} \\[3mm] 
		\infer[\mathrm{(\land \texttt{l2})}]{(\phi \land \psi) \Rightarrow \gamma}{\psi \Rightarrow \gamma}  & \infer[\mathrm{(\land \texttt{r})}]
		{\Gamma \Rightarrow (\phi \land \psi)}{\Gamma \Rightarrow \phi & \quad \Gamma \Rightarrow \psi}   &
		\infer[\mathrm{(\lor \texttt{l1})}] 
		{(\phi \lor \psi) \Rightarrow \gamma }{\phi \Rightarrow \gamma  \quad \psi \Rightarrow \gamma} \\[3mm] 
		\infer[\mathrm{(\lor \texttt{l2})}] 
		{((\phi \lor \psi) \Rightarrow^2 \gamma) }{\phi \Rightarrow^2 \gamma  & & \psi \Rightarrow^2 \gamma} 
		\hspace{2mm}
		& 
		\infer[\mathrm{(\lor \texttt{r1} )}]{\Gamma \Rightarrow (\phi \lor \psi)}{\Gamma \Rightarrow \phi}  & \infer[\mathrm{(\lor \texttt{r2})}]{\Gamma \Rightarrow (\phi \lor \psi)}{\Gamma \Rightarrow \psi} \\[3mm] 
		\infer[\mathrm{(\neg \Rightarrow \texttt{l})}]{\neg(\phi \Rightarrow \psi) \Rightarrow \gamma}{(\phi \land \neg \psi) \Rightarrow \gamma} 
		& 
		\infer[\mathrm{(\neg \Rightarrow \texttt{r})}]{\Gamma \Rightarrow^{2} \neg (\phi \Rightarrow \psi)}{\Gamma \Rightarrow^{2} (\phi \land \neg \psi)}  
		\hspace{2mm}
		&
		\infer[\mathrm{(\neg \land \texttt{l})}]{\neg(\phi \land \psi) \Rightarrow \gamma }{(\neg \phi \lor \neg \psi) \Rightarrow \gamma} \\[3mm] 
		\infer[\mathrm{(\neg \land\texttt{r})}]{\Gamma \Rightarrow \neg (\phi \land \psi)}{\Gamma \Rightarrow (\neg \phi \lor \neg \psi)}  &
		\infer[\mathrm{(\neg \lor \texttt{l})}]{\neg (\phi \lor \psi) \Rightarrow \gamma}{(\neg \phi \land \neg \psi) \Rightarrow \gamma} &  \infer[\mathrm{(\neg\lor\texttt{r})}]{\Gamma \Rightarrow \neg(\phi \lor \psi)}{\Gamma \Rightarrow (\neg \phi \land \neg \psi)}  \\[3mm]
		\infer[\mathrm{(\neg \neg \texttt{l})}]{\neg \neg \phi \Rightarrow \gamma}{\phi \Rightarrow \gamma} & \infer[\mathrm{(\neg \neg\texttt{r})}]{\Gamma \Rightarrow \neg \neg \phi}{\Gamma \Rightarrow \phi} & 
		\\
	\end{tabular}
}
\medskip

In the above rules, following Nelson's notation, $\Gamma = \{ \phi_1, \phi_2, \ldots, \phi_n \}$ is a finite set of formulas and the following abbreviations are employed:
\smallskip

\noindent
$\Gamma \Rightarrow \phi \ := \phi_{1} \Rightarrow (\phi_{2} \Rightarrow (\ldots  \Rightarrow (\phi_{n} \Rightarrow \phi)\ldots))$\\
%
$\phi \Rightarrow^{2} \psi := \phi \Rightarrow (\phi \Rightarrow \psi)$\\
$\Gamma \Rightarrow^{2} \phi := \phi_{1} \Rightarrow^{2} (\phi_{2} \Rightarrow^{2}( \ldots  \Rightarrow^{2} (\phi_{n} \Rightarrow^{2} \phi)\ldots))$\smallskip
\\
Moreover, 
when $\Gamma = \emptyset$, we take $\Gamma \Rightarrow \phi := \phi$.

Notice that we have fixed obvious infelicities in the rules
$\mathrm{(\land \texttt{l2})}$, $\mathrm{(\land \texttt{r})}$ and  $\mathrm{(\neg \Rightarrow \texttt{r})}$ as they appear in \cite[pp.214--5]{Nel59}. For example,  the original rule $\mathrm{(\land \texttt{l2})}$ in Nelson's paper was:
$$\infer[\mathrm{(\land \texttt{l2})}]{\psi \Rightarrow \gamma}{(\phi \land \psi) \Rightarrow \gamma}$$
This clearly makes the logic inconsistent. Indeed, taking $\phi = \gamma$, we have:
$$\infer[\mathrm{(\land \texttt{l2})}]{\psi \Rightarrow \gamma}{(\gamma \land \psi) \Rightarrow \gamma}$$
Now, since $(\gamma \land \psi) \Rightarrow \gamma$ is a theorem (see Prop.~\ref{prop1}, below), $\psi \Rightarrow \gamma$ is a theorem too. Choosing~$\psi$ as an axiom, we would conclude thus that~$\gamma$ is a theorem for any formula~$\gamma$. 

We note in passing that the rule \texttt{(C)}, called \emph{weak condensation} by Nelson, replaces (and is indeed a weaker form of) the usual contraction rule: 
$$\infer{\phi \Rightarrow \psi}{\phi \Rightarrow (\phi \Rightarrow \psi)}$$
Rule \texttt{(C)} is also known  in the literature as \emph{3-2 contraction}
\cite[p.~389]{Re93c}
and 
corresponds,
on  algebraic models, 
to the 
property of \emph{three-potency} (see Subsection~\ref{ss:alt}).

\noindent Also, do note that we obtain \textit{modus ponens}, $\mathrm{\texttt{(MP)}}$, by taking $\Gamma = \emptyset$ in rule $\mathrm{\texttt{(E)}}$:
$$\infer{\gamma }{\phi  \quad \phi \Rightarrow \gamma}$$

It is worth noticing that, despite appearances, Nelson's system~$\NS$ is a Hilbert-style calculus, rather than a sequent system.  Its underlying notion of derivation, $\vdash_\NS$, is the usual one.  
Henceforth, for any logic~$\mathcal{L}$ and any set of formulas~$\Gamma\cup\Pi$, we shall write $\Gamma\vdash_\mathcal{L}\Pi$ to say that $\Gamma\vdash_\mathcal{L}\pi$ for every $\pi\in\Pi$.  By $\Gamma\dashv\vdash_\mathcal{L}\Pi$ we will abbreviate the double assertion $\Gamma\vdash_\mathcal{L}\Pi$ and $\Pi\vdash_\mathcal{L}\Gamma$.

One of  the crucial steps in proving that a logic is algebraizable (in the sense of Blok and Pigozzi \cite[Definition 2.2]{BP89}) is to prove that it satisfies certain congruence properties.
In the present context, this entails checking that $\phi \Leftrightarrow \psi \vdash_{\NS} \neg \phi \Leftrightarrow \neg \psi$ and 
$\{\phi_1 \Leftrightarrow \psi_1,  \phi_2 \Leftrightarrow \psi_2\} \vdash_{\NS} (\phi_1 \bullet \phi_2) \Leftrightarrow (\psi_1 \bullet \psi_2)$ for each  connective $ \bullet \in \{ \land, \lor, \Rightarrow\}$. The following auxiliary results will be used to prove that much, in the next section.

\begin{proposition} 	
	\label{p:ths}
	The following formulas are theorems of~$\NS$:
	\begin{enumerate}
		\label{prop1}
		\item  $(\phi \land \psi) \Rightarrow \phi$
		\item $(\phi \land \psi) \Rightarrow \psi$
		\item $\phi \Rightarrow ( \phi \lor \psi)$
		\item  $\psi \Rightarrow ( \phi \lor \psi)$
		\item $(\phi \Rightarrow (\psi \Rightarrow \gamma)) \Leftrightarrow (\psi \Rightarrow (\phi \Rightarrow \gamma))$
	\end{enumerate}
\end{proposition} 

\begin{proof}
	All justifying derivations are straightforward. 
	We detail the first item, as an example:
	$$
	\infer[\mathrm{(\land \texttt{l1})}]{(\phi\land\psi) \Rightarrow \phi}{\infer[\texttt{(A1)}]{\phi \Rightarrow \phi}{}}
	$$
\end{proof}

\begin{proposition} 	
	\label{p:eqv}
	$\{\varphi\Leftrightarrow\psi\} \dashv\vdash_\NS \{\varphi\Rightarrow\psi,\psi\Rightarrow\varphi\}$.
\end{proposition} 	
\begin{proof} 	
	Such a logical equivalence is easily justified by Prop.\ref{p:ths}.1--2 and by considering the rule $\mathrm{(\land \texttt{r})}$ with $\Gamma=\emptyset$.
\end{proof}

\section{$\NS$ is algebraizable} \label{section3}

In this section we prove that~$\NS$ is algebraizable in the sense of Blok and Pigozzi (it is, in fact, implicative \cite[Definition 2.3]{Font16}), and we give two alternative presentations for its equivalent algebraic semantics (to be called `$\NS$-algebras'). 
The first one is obtained via the algorithm of \cite[Theorem 2.17]{BP89}, while the second one is closer to the usual axiomatizations of classes of residuated lattices, which are the algebraic counterparts of many logics in the substructural family. As a particular advantage, the second presentation of~$\NS$-algebras will allow us to see at a glance that they form an equational class, and will also make it easier to compare them with other known classes of algebras for substructural logics.

\begin{definition}
	\label{def:imp}
	An \emph{implicative logic} is a sentential logic $\mathcal{L}$  
	whose underlying algebra of formulas in a language~$\Sigma$ has a term $\alpha(p,q)$ in two variables that satisfies
	the following conditions:\smallskip\\
	\begin{tabular}{ll}
		{\rm [IL1]\ \ } & $\vdash_{\mathcal{L}} \alpha(p, p) $ \\
		{\rm [IL2]} & $\alpha(p,q), \alpha(q, r) \vdash_{\mathcal{L}} \alpha(p, r)$ \\
		{\rm [IL3]} & $ p, \alpha(p,q) \vdash_{\mathcal{L}} q$ \\
		{\rm [IL4]} & $q \vdash_{\mathcal{L}} \alpha(p,q)$ \\
		{\rm [IL5]} & for each $n$-ary 
		$\bullet\in\Sigma$, \\
		& $\bigcup_{i=1}^n \{ \alpha(p_{i}, q_{i}), \alpha(q_{i}, p_{i}) \} \vdash_{\mathcal{L}} \alpha(\bullet (p_{1}, \ldots, p_{n}), \bullet(q_{1}, \ldots, q_{n}))$
	\end{tabular}\smallskip\\
	We call any such $\alpha$ an \emph{$\mathcal{L}$-implication}.
\end{definition}

Given an algebra of formulas $\mathbf{Fm}$ of the language~$\Sigma$, the associated set $Fm \times Fm$ of \emph{equations} will henceforth be denoted by $\mathsf{Eq}$; we will write $\phi \approx \psi$ rather than $( \phi, \psi )\in\mathsf{Eq}$. Let $\mathbf{A}$ be an algebra with the same similarity type as $\mathbf{Fm}$.
A~homomorphism $\V \colon \mathbf{Fm} \to\mathbf{A}$ is called a \emph{valuation in~$\mathbf{A}$}.
We say that a valuation~$\V$ in~$\mathbf{A}$ \emph{satisfies $\phi \approx  \psi$ in $\mathbf{A}$} when $\V(\phi) = \V(\psi)$; we say that an algebra $\mathbf{A}$ \emph{satisfies} $\phi \approx  \psi$ when all valuations in~$\mathbf{A}$ satisfy it.

\begin{definition}	
	A logic $\mathcal{L}$ is \emph{algebraizable} if and only if there are equations 
	$\mathsf{E}(\phi) \subseteq \mathsf{Eq}$ and a transform $\mathsf{Eq} \stackrel{\rho}{\longmapsto} 2^{Fm}$, denoted by $\Delta(\phi, \psi) := \rho(\phi \approx \psi)$,
	such that
	$\mathcal{L}$ respects the following conditions: \smallskip\\
	\begin{tabular}{ll}
		{\rm [Alg]} & $\phi \dashv \vdash_{\mathcal{L}} \Delta(\mathsf{E}(\phi))$ \\
		{\rm [Ref]} & $\vdash_{\mathcal{L}} \Delta(\phi, \phi)$ \\
		{\rm [Sym]} & $\Delta(\phi, \psi) \vdash_{\mathcal{L}} \Delta(\psi, \phi)$ \\
		{\rm [Trans]\ \ } & $\Delta(\phi, \psi) \cup \Delta(\psi, \gamma) \vdash_{\mathcal{L}} \Delta(\phi, \gamma)$ \\
		{\rm [Cong]} & for each $n$-ary $\bullet \in \Sigma$, \\
		& $\bigcup_{i=1}^n \Delta(\phi_{i}, \psi_{i}) \vdash_{\mathcal{L}} \Delta(\bullet(\phi_{1}, \cdots, \phi_{n}), \bullet(\psi_{1}, \cdots, \psi_{n}))$ \\
	\end{tabular}\smallskip\\	
	We call any such $\mathsf{E}(\phi)$ the set of \emph{defining equations} and any such $\Delta(\phi, \psi)$ the set of \emph{equivalence formulas} of~$\mathcal{L}$.	
\end{definition}

\noindent 
Clarifying the notation in [Alg], recall that the set $\mathsf{E}(\phi)$ contains pairs of formulas and we write $\phi \approx \psi$ simply as syntactic sugar for a pair~$(\varphi,\psi)$ belonging to this set. Now, $\Delta(\phi, \psi)$ transforms an equation into a set of formulas.  Accordingly, we take $\Delta(\mathsf{E}(\phi))$ as $\bigcup\{\Delta(\varphi_1,\varphi_2)\mid(\varphi_1,\varphi_2)\in\mathsf{E}(\phi)\}$.
Similarly, we shall let
$\mathsf{E}(\Delta(\phi, \psi))$
stand for 
$\bigcup\{\mathsf{E}(\chi)\mid \chi \in \Delta(\phi, \psi)\}$.

\begin{definition} Let $\mathcal{L}$ be an implicative logic in the language~$\Sigma$, having an $\mathcal{L}$-im\-plication~$\alpha$. An \emph{$\mathcal{L}$-algebra}~$\NA$ is a $\Sigma$-algebra such that $1\in A$ and:
	\begin{description}
		\item[{[LALG1]}] For all $\Gamma \cup \{\phi\} \subseteq Fm$ and every valuation~$\V$ in~$\NA$, \\
		if $\Gamma \vdash_{\mathcal{L}} \phi$ and $\V (\Gamma) \subseteq \{1\}$, then $\V(\phi) = 1$.
		\item[{[LALG2]}] For all $a, b \in A$, 
		if $\alpha(a, b) = 1$ and $\alpha(b, a) = 1$, then $a = b$.
	\end{description}
	\noindent The class of $\mathcal{L}$-algebras is denoted by $Alg^*\mathcal{L}$.
\end{definition}

Every implicative logic $\mathcal{L}$ is algebraizable with respect to the class $Alg^*\mathcal{L}$ \cite[Proposition 3.15]{Font16}, and such algebraizability is 
witnessed
by the defining equations $\mathsf{E}(\phi) := \{ \phi \approx \alpha(\phi, \phi) \}$ 
and the equivalence formulas  $\Delta(\phi, \psi) := \{ \alpha(\phi, \psi),  \alpha(\psi, \phi) \}$. 
These are in fact the sets of defining equations and of equivalence formulas that we will use in the remainder of the present paper. 

We can now prove (the details are to be found in the Appendix) that:

\begin{theorem}
	\label{alg}
	The calculus $\vdash_{\NS}$ is implicative
	and thus algebraizable.	
	The $\NS$-implication is given by $\Rightarrow$, that is, $\alpha(p, q) := p \Rightarrow q$.
\end{theorem} 

In the case of $\NS$ we have thus  that $\mathsf{E}(\phi) = \{ \phi \approx \phi \Rightarrow \phi \}$ 
and $\Delta(\phi, \psi) = \{ \phi \Rightarrow \psi, \psi \Rightarrow \phi \}$.

\subsection{$\NS$-algebras}

By Blok-Pigozzi's algorithm (\cite[Theorem 2.17]{BP89}, see also~\cite[Proposition 3.44]{Font16}), we know that the equivalent algebraic semantics of~$\NS$ is the class of algebras given by Def.~\ref{def:salg} below.
We denote by $\mathsf{Ax}$ the set of axioms and denote by $\mathsf{Inf\,R}$ the set the inference rules of~$\NS$, given in Def.~\ref{ax:S}. 

\begin{definition}
	\label{def:salg}
	
	An \emph{$\NS$-algebra} is a structure $\mathbf{A} = \langle A, \land, \lor, \Rightarrow, \neg, 0, 1 \rangle$ of type $\langle 2, 2, 2, 1, 0, 0\rangle$ that satisfies the following equations and quasiequations:
	\begin{enumerate}
		
		\item 	$\mathsf{E}(\Delta(\varphi, \varphi))$
		
		\item  $ \mathsf{E}(\Delta(\varphi,\psi)) $ implies $  \varphi \approx \psi$
		
		\item  
		$\mathsf{E}(\varphi)$, for each $\phi \in \mathsf{Ax}$
		
		\item $\bigcup\limits_{i=1}^{n}\mathsf{E}(\gamma_{i})$ implies $\phi \approx 1$  for each 
		$\ \gamma_{1}, \cdots, \gamma_{n} \vdash_{\NS} \phi \in \mathsf{Inf\,R}$
		
	\end{enumerate}
\end{definition}

\noindent 
Regarding the notation in the above definition, $\mathsf{E}(\Delta(\varphi, \varphi))$ 
stands for 
the equation $\phi \Rightarrow \phi \approx  (\phi \Rightarrow \phi) \Rightarrow (\phi \Rightarrow \phi)$. 
Item 2 
is the quasiequation:  $(\phi \Rightarrow \psi) \approx (\phi \Rightarrow \psi) \Rightarrow (\phi \Rightarrow \psi)$ and $ (\psi \Rightarrow \phi) \approx  (\psi \Rightarrow \phi) \Rightarrow (\psi \Rightarrow \phi)$  implies $\phi \approx \psi$; $\mathsf{E}(\varphi)$ 
is the equation $\phi \approx \phi \Rightarrow \phi$
for each axiom~$\phi$ of $\NS$. In fact, these conditions are telling us that for each axiom~$\phi$ of $\NS$ we have the equation $\phi \approx 1$, and for each rule  $\phi \vdash_{\NS} \psi$ of  $\NS$ , in the corresponding algebras we have the quasiequation: if $ \phi \approx 1$, then $ \psi \approx 1$.

We shall denote by $\mathsf{E}(\texttt{An})$ the equation given in Def.~\ref{def:salg}.3 for the axiom $\texttt{An}$ (for $ 1 \leq \texttt{n} \leq 5 $), and by $\mathsf{Q}\texttt{(R)}$ the quasiequation given in Def.~\ref{def:salg}.4 for the rule $\texttt{R}$ of~$\NS$.
From this point on, in this subsection, in order to make the propositions and their proofs shorter,  we shall also use the following abbreviations: \smallskip \\
\begin{tabular}{l}
	\quad $x * y := \neg (x \Rightarrow \neg y)$ \\
	\quad $x^2 := x * x$\\
	\quad $x^n : = x * (x^{n-1})$, for $n > 2$
\end{tabular}
\smallskip

\noindent 
The following result, whose proof may be found in the Appendix, will help us in checking that the class of $\NS$-algebras forms a variety.
\begin{proposition} Let  $\mathbf{A}$  be an~$\NS$-algebra and let $a,b, c \in A$. Then:
	\label{prop:coisas}
	\begin{enumerate}
		\item 
		$a \Rightarrow a = 
		1 = \neg 0$. 
		
		\item The relation $\leq$  defined by setting
		$a \leq b $ iff $a \Rightarrow b = 1 $, 
		is a partial order with maximum~$1$ and
		minimum $0$. 
		
		\item $  a \Rightarrow b = \neg b \Rightarrow \neg a$.
		
		\item $a \Rightarrow (b \Rightarrow c) = b \Rightarrow (a \Rightarrow c)$.

		\item $\neg \neg a = a $ and $a \Rightarrow 0 = \neg a $.
		
		\item $\langle A,*,1\rangle$ is a commutative monoid.
		
		\item $(a * b) \Rightarrow c = a \Rightarrow (b \Rightarrow c)$.

		\item 	The pair $( *, \Rightarrow )$ is residuated with respect to $\leq$, i.e., 
		\
		$
		a * b \leq c \  \textrm{ iff } \ b \leq a \Rightarrow c.
		$ 

		\item $a^2 \leq a^3$.
		\item $\langle A, \land, \lor \rangle$ is a lattice with order $\leq$.

		\item $(a \lor b)^2 \leq a^2 \lor b^2$.
		
	\end{enumerate}
\end{proposition}

In the next section we introduce an equivalent presentation of~$\NS$-algebras which takes precisely the properties of Prop.~\ref{prop:coisas} above as postulates.

\subsection{Alternative presentation of~$\NS$-algebras}
\label{ss:alt}
We start here by recalling the following  standard definition~\cite[p.185]{GaJiKoOn07}:

\begin{definition}\label{CIBRL-def}
	A \emph{commutative integral bounded residuated lattice (CIBRL)} is an algebra $\mathbf{A} = \langle A, \land, \lor, *, \Rightarrow, 0, 1 \rangle$ of type $\langle2 ,2 ,2, 2, 0, 0 \rangle$
	such that:
	\begin{enumerate}
		\item $\langle A, \land, \lor, 0, 1 \rangle$  is a bounded lattice with ordering $ \leq$, minimum element~$0$ and maximum element~$1$.
		
		\item $\langle A, *, 1 \rangle$ is a commutative monoid.
		
		\item The pair $( *, \Rightarrow )$ is residuated with respect to $\leq$, i.e., 
		\
		$
		a * b \leq c \  \textrm{ iff } \ b \leq a \Rightarrow c.
		$
	\end{enumerate}
\end{definition}
In the context of the above definition, the \emph{integrality} condition corresponds to having~$1$ not only as a maximum but also as the multiplicative unit of the operation~$*$, that is, $ x*1 = x$. 
For a CIBRL this condition immediately follows from Def.~\ref{CIBRL-def}.1--2.

Setting $\neg x := x \Rightarrow 0$, we  say that a residuated lattice is \emph{involutive} \cite[p.186]{GaRa04} when $\neg \neg a = a$ (in such a case, it follows that $ a \Rightarrow b = \neg b \Rightarrow \neg a)$. We say that a residuated lattice is \emph{3-potent} when it satisfies the equation $x^2 \leq x^3$.
While we have earlier defined~$*$ from $\Rightarrow$, and now~$*$ is a primitive operation, we can show that every CIBRL satisfies $x * y = \neg (x \Rightarrow \neg y)$ (see \cite[Lemma 5.1]{GaRa04}).

\begin{definition}
	\label{def:slinha}
	An \emph{$\NS^{\prime}$-algebra} is an involutive 3-potent CIBRL.
\end{definition}

The proof of the following result may be found in the Appendix:

\begin{lemma} \label{lemma1} 
	\begin{enumerate}
		\item	Any CIBRL 
		satisfies the equation $ (x\lor y)*z \approx (x*z) \lor (y*z)$.
		\item Any CIBRL 
		satisfies 
		$x^{2} \lor y^{2} \approx (x^{2} \lor y^{2})^{2}$.
		\item  Any 3-potent CIBRL 
		satisfies 
		$(x\lor y^{2})^{2} \approx (x\lor y)^{2}.$
		\item Any 3-potent CIBRL 
		satisfies  $( x\lor y)^{2} \approx x^{2} \lor y^{2}$.
	\end{enumerate}
\end{lemma}

Since involutive residuated lattices form an equational class \cite[Theorem 2.7]{GaJiKoOn07}, it is obvious that
$\NS^{\prime}$-algebras are also an equational class. 
From Prop.~\ref{prop:coisas}, we immediately conclude the following:

\begin{proposition}
	\label{prop:slinha}
	Let $\mathbf{A} = \langle A, \land, \lor, \Rightarrow, \neg, 0, 1 \rangle$ be an~$\NS$-algebra. Defining $x*y := \neg (x \Rightarrow \neg y)$, 
	we have that  $\mathbf{A'} = \langle A, \land, \lor, *, \Rightarrow,  0, 1 \rangle$ is an $\NS^{\prime}$-algebra.
\end{proposition}

\noindent
Conversely, we are going to see that every $\NS^{\prime}$-algebra gives rise to an~$\NS$-algebra by checking that
all (quasi) equations introduced in Definition~\ref{def:salg} are satisfied (the proof may be found in the Appendix):

\begin{proposition} \label{equivalence}
	Let $\mathbf{A} = \langle A, \land, \lor, *, \Rightarrow, 0, 1 \rangle$ be an $\NS^{\prime}$-algebra. Defining $\neg x := x \Rightarrow 0$, 
	we have that  $\mathbf{A'} = \langle A, \land, \lor, \Rightarrow,  \neg, 0, 1 \rangle$ is an~$\NS$-algebra.
\end{proposition}

\noindent 
Thus,  the classes of ~$\NS$-algebras and of $\NS^{\prime}$-algebras are term-equivalent.

The presentation given in Definition~\ref{def:slinha} has several advantages in what concerns the study of the semantics of~$\NS$. For example, it is now straightforward to check that the three-element MV-algebra~\cite{CiMuOt99} is a model of Nelson's logic~$\NS$.
This in turn allows one to prove that the formulas which Nelson claims not to be derivable in~$\NS$~\cite[p.213]{Nel59} are indeed not valid (see~\cite{Igpl}).

\section{A finite Hilbert-style calculus for~$\NS$}
\label{section4} 
In this section we  introduce a finite Hilbert-style calculus (which is an extension of the calculus $IPC^{*}\backslash c$, called \emph{intuitionistic logic without contraction}, of \cite{BoGaV06})  that 
is algebraizable with respect to the class of $\NS^{\prime}$-algebras.

We are thus going to have two logics that are both algebraizable with respect to the same variety with the same defining equations and equivalence formulas; from this we will obtain an equivalence between our calculus and Nelson's. 

The logic $\NS^{\prime} = \langle \mathbf{Fm}, \vdash_{\NS^{\prime}} \rangle$ is the sentential logic in the language $\langle \land, \lor, \Rightarrow, \linebreak *, \neg, \bot, \top \rangle $ of type $\langle 2, 2, 2, 2, 1, 0, 0 \rangle $ defined by the Hilbert-style calculus with the following schematic axioms and with \textit{modus ponens} as the only rule:
\smallskip

\noindent 
{\small
	$(\texttt{A1'}) \ (\varphi \Rightarrow \psi) \Rightarrow ((\gamma \Rightarrow \varphi) \Rightarrow (\gamma \Rightarrow \psi)) $	\\
	$(\texttt{A2'}) \quad (\varphi \Rightarrow (\psi \Rightarrow \gamma )) \Rightarrow (\psi \Rightarrow (\varphi \Rightarrow \gamma ))$  \\
	$(\texttt{A3'}) \quad \varphi \Rightarrow (\psi \Rightarrow \varphi)$  \\
	$(\texttt{A4'}) \quad (\varphi\Rightarrow\gamma )\Rightarrow((\psi \Rightarrow\gamma )\Rightarrow((\varphi \lor \psi) \Rightarrow\gamma ))$  \\
	$(\texttt{A5'}) \quad \varphi \Rightarrow (\varphi \lor \psi) $	  \\
	$(\texttt{A6'}) \quad \psi \Rightarrow (\varphi \lor \psi) $\\
	$(\texttt{A7'}) \quad (\varphi \land\psi) \Rightarrow \varphi $	 \\
	$(\texttt{A8'}) \quad (\varphi \land\psi) \Rightarrow \psi$	 \\
	$(\texttt{A9'}) \quad \varphi \Rightarrow (\psi \Rightarrow (\varphi \land\psi))$ \\
	$(\texttt{A10'}) \quad ((\gamma \Rightarrow\varphi) \land(\gamma \Rightarrow\psi))\Rightarrow(\gamma \Rightarrow (\varphi \land\psi))$ \\
	$(\texttt{A11'}) \quad \varphi \Rightarrow (\psi \Rightarrow (\varphi * \psi))$ \\
	$(\texttt{A12'}) \quad (\varphi \Rightarrow (\psi \Rightarrow \gamma )) \Rightarrow ((\varphi * \psi) \Rightarrow \gamma )$ \\
	$ (\texttt{A13'}) \quad \neg\varphi \Rightarrow (\varphi \Rightarrow \psi)$ \\
	$(\texttt{A14'}) \quad (\varphi \Rightarrow \psi) \Leftrightarrow (\neg\psi \Rightarrow \neg\varphi)$ \\
	$(\texttt{A15'}) \quad \varphi \Leftrightarrow \neg\neg\varphi$  \\
	$(\texttt{A16'}) \quad \bot \Rightarrow \varphi$ \\
	$(\texttt{A17'}) \quad \varphi \Rightarrow \top$ \\
	$(\texttt{A18'}) \quad \phi^2 \Rightarrow \phi^3 $
}
\smallskip
\\
As before, $\phi \Leftrightarrow \psi$ abbreviates $(\phi \Rightarrow \psi ) \land (\psi \Rightarrow \phi)$, while the connective~$*$ is here taken as primitive.

Axioms (\texttt{A1'})--(\texttt{A13'}), $(\texttt{A14'}(\Rightarrow))$ - $(\varphi \Rightarrow \psi) \Rightarrow (\neg\psi \Rightarrow \neg\varphi)$, $(\texttt{A15'}(\Rightarrow))$ - $\varphi \Rightarrow \neg\neg\varphi$, $(\texttt{A16'})$ and $(\texttt{A17'})$ of our 
calculus 
are the same as those of $IPC^{*}\backslash c$ as presented in  \cite[Table 3.2]{BoGaV06}, where 
it is proven that
$IPC^{*}\backslash c$ is 
algebraizable.
We added the converse implication in axioms (\texttt{A14'}) and (\texttt{A15'}) to characterize  involution and we added the axiom (\texttt{A18'}) to characterize 3-potency. As
algebraizability is preserved by axiomatic extensions (cf.~\cite[Proposition 3.31]{Font16}) we have the following results:

\begin{theorem}
	The calculus $\NS^{\prime}$ is algebraizable (with the same defining equations and equivalence formulas as~$\NS$) with respect to the class of
	$\NS^{\prime}$-algebras.
\end{theorem}
\begin{proof} We know from
	\cite[Theorem 5.1]{BoGaV06} that $IPC^{*}\backslash c$ is algebraizable with respect to the class of commutative integral bounded residuated lattices with the same defining equations and equivalence formulas already considered above. The axioms that were now added imply that the algebraic semantics of our extension is involutive and 3-potent, i.e., it is an~$\NS^{\prime}$-algebra. 
\end{proof}

\begin{corollary} 
	$\NS$ and $\NS^{\prime}$ define the same logic.
\end{corollary}

\begin{proof} Let~$K_\NS$ be the class of $\NS$-algebras. Thanks to Prop.~\ref{prop:slinha} and Prop.~\ref{equivalence} we know that~$K_\NS$ is also the class of $\NS^{\prime}$-algebras. The result follows now from \cite[Proposition 3.47]{Font16}, that gives us an algorithm to find a Hilbert-style calculus for an algebraizable logic from its quasivariety, defining equations and equivalence formulas. As $\NS$-algebras and $\NS^{\prime}$-algebras are the same class of algebras and their defining equations and equivalence formulas are the same, the Hilbert-style calculus given by the algorithm must do the same job as the one we had before.
\end{proof}


Working with Nelson's original presentation of~$\NS$, it can be hard to directly prove some version of the Deduction Metatheorem. Indeed, if we prove it, as usual, by way of an induction over the structure of the derivations, we need to apply the inductive hypothesis over each rule of the system. The advantage of~$\NS^{\prime}$, in employing such a strategy, is that it has only one inference rule. 
This allows us to establish:

\begin{theorem}[Deduction Metatheorem]\label{deduction} If $\Gamma \cup \{\varphi\} \vdash \psi$, then  $\Gamma \vdash \varphi^{2} \Rightarrow \psi$.
\end{theorem}

\begin{proof} 
	Thanks to \cite[Corollary 2.15]{GaJiKoOn07} we have a version of the Deduction Metatheorem for  substructural logics wich says that  $\Gamma \cup \{\varphi\} \vdash \psi$ iff   $\Gamma \vdash \varphi^{n} \Rightarrow \psi$ for some~$n$. In view of  (\texttt{A'18}) it is easy to see that in~$\NS$ we can always choose $n=2$.
\end{proof}

\section{Comparing~$\NS$ with $\mathcal{N}3$ and $\mathcal{N}4$} \label{section5}

As mentioned before, Nelson introduced two other better-known logics, $\mathcal{N}3$ and $\mathcal{N}4$, which are also algebraizable with respect to classes of residuated structures (namely, the so-called $\mathcal{N}3$-lattices and $\mathcal{N}4$-lattices). A question that immediately arises concerns the precise relation between~$\NS$ and these other logics, or (equivalently) between~$\NS$-algebras and 
$\mathcal{N}3$- and $\mathcal{N}4$-lattices. 
In what follows it is worth taking into account that not all $\NS$-algebras are distributive (see \cite[Example 5.1]{Igpl}).

\subsection{$\mathcal{N}4$}
\begin{definition}
	$\mathcal{N}4 = \langle \mathbf{Fm}, \vdash_{\mathcal{N}4} \rangle$ is 
	the sentential logic in the language $\langle \land, \lor, \linebreak \to, \neg \rangle $ of type $\langle 2, 2, 2, 1 \rangle$ defined by the Hilbert-style calculus with the following schematic axioms and \textit{modus ponens} 
	as the only schematic rule.  Below, $\phi\leftrightarrow\psi$ will be used to abbreviate $(\phi\rightarrow\psi)\land(\psi\rightarrow\phi)$.
	\begin{description}
		\item[\texttt{(N1)}] \qquad $\phi \to (\psi \to \phi)$
		\item[\texttt{(N2)}] \qquad $(\phi \to (\psi \to \gamma)) \to ((\phi \to \psi) \to (\phi \to \gamma))$
		\item[\texttt{(N3)}]\qquad  $(\phi \land \psi) \to \phi$
		\item[\texttt{(N4)}] \qquad $(\phi \land \psi) \to \psi$
		\item[\texttt{(N5)}] \qquad $(\phi \to \psi) \to ((\phi \to \gamma) \to (\phi \to (\psi \land\gamma)))$
		\item[\texttt{(N6)}] \qquad $\phi \to (\phi \lor \psi)$
		\item[\texttt{(N7)}] \qquad $\psi \to (\phi \lor \psi)$
		\item[\texttt{(N8)}] \qquad $(\phi \to \gamma) \to ((\psi \to \gamma) \to ((\phi \lor \psi) \to \gamma))$
		\item[\texttt{(N9)}] \qquad $\neg \neg \phi \leftrightarrow \phi$
		\item[\texttt{(N10)}] \quad $\neg (\phi \lor\psi) \leftrightarrow (\neg \phi\land\neg \psi)$
		\item[\texttt{(N11)}] \quad $\neg (\phi \land\psi) \leftrightarrow (\neg \phi\lor\neg \psi)$
		\item[\texttt{(N12)}] \quad $\neg (\phi \to \psi) \leftrightarrow (\phi\land\neg \psi)$
	\end{description}
\end{definition}

The implication $\to$ in $\mathcal{N}4$ is usually called \emph{weak implication}, in contrast to the \emph{strong implication} $\Rightarrow$ that is defined in $\mathcal{N}4$ as follows: 
$$
\phi \Rightarrow \psi :=  (\phi \to \psi) \land (\neg \psi \to \neg \phi).
$$
As the notation suggests, it is the strong implication that we shall compare with the implication of~$\NS$.
This appears indeed to be the most meaningful choice, for otherwise, since the weak implications of
both $\mathcal{N}4$ and $\mathcal{N}3$ fail to satisfy contraposition (which holds in~$\NS$), 
we would have to say that~$\NS$
is incomparable with both logics. 

The logic $\mathcal{N}4$ is algebraizable (though not implicative) with equivalence formulas 
$\{ \phi \Rightarrow \psi, \psi \Rightarrow \phi \}$
and defining equation $\phi \approx \phi \to \phi$ \cite[Theorem 2.6]{Ri11c}.
We notice in passing that the implication in this defining equation could as well be taken to be the strong one,
so $\phi \approx \phi \Rightarrow \phi$ would work too; in contrast, 
$\{ \phi \to \psi, \psi \to \phi \}$ 
would not be a set of equivalence formulas,
due precisely to the failure of contraposition.
The equivalent algebraic semantics of $\mathcal{N}4$ is the class of $\mathcal{N}4$-lattices defined below
\cite[Definition 8.4.1]{Od08}:

\begin{definition} 
	\label{def:n4}
	An algebra $\NA$ =  $\langle A, \lor, \land, \to, \neg  \rangle$ is an \emph{$\mathcal{N}4$-lattice} if it satisfies the following properties:
	
	\begin{description}
		\item[1]  $\langle A, \lor, \land, \neg  \rangle$ is a De Morgan algebra.
		\item[2] The relation  $\preceq$ defined, for all $a, b \in A$, by $a \preceq b$ iff $(a \to b) \to (a \to b) = (a \to b)$ is a pre-order on $\NA$.
		\item[3] The relation $\equiv $ defined, for all $a, b \in A$ as $a \equiv b$  iff $a \preceq b$ and $b \preceq a$ is a congruence relation with respect to $\land, \lor, \to$ and the quotient algebra $\NA_{\bowtie} := \langle A, \lor, \land, \to  \rangle$/$\equiv$ is an implicative lattice. 
		\item[4] For any $a, b \in A$, $\neg(a \to b) \equiv a  \land \neg b$.
		\item[5] For any $a, b \in A$, $a \leq b$ iff $a \preceq b$ and $\neg b \preceq \neg a$, where $\leq$ is the lattice order for~$\NA$.
		
	\end{description}
\end{definition}

A very simple example of an $\mathcal{N}4$-lattice is the four-element algebra~$\mathbf{A_4}$ whose lattice reduct is the four-element
diamond De Morgan algebra. This algebra has carrier $A_4 = \{0, 1, b, n  \}$, the maximum element of the lattice order
being 1, the minimum 0, and $b$ and~$n$ being incomparable. The negation (Fig 1) is given by $\neg b := b $, $\neg n := n$,
$\neg 1 := 0$ and $\neg 0 := 1$. The weak implication is given, for all $a \in A_4$, by $1 \to a = b \to a := a $ and
$0 \to a = n \to a := 1$. One can check that $\mathbf{A_4}$ satisfies all properties of Definition~\ref{def:n4} (in particular, the quotient  $\mathbf{A_4}$/$\equiv$ is the two-element Boolean algebra). 

\begin{proposition} 
	$\mathcal{N}4$ and~$\NS$ are incomparable, that is, neither of them extends the other.
\end{proposition}

\begin{proof} We show that not every~$\NS$-algebra is an $\mathcal{N}4$-lattice, and that no $\mathcal{N}4$-lattice is an~$\NS$-algebra.
	The first claim follows from the fact that $\mathcal{N}4$-lattices have a distributive lattice reduct, whereas~$\NS$-algebras need not be distributive. As to the second, it is sufficient to observe that the equation $x \Rightarrow x \approx y \Rightarrow y$ is satisfied in all~$\NS$-algebras but does not hold in the four-element $\mathcal{N}4$-lattice $\mathbf{A_4}$.
	There we have $1 \Rightarrow 1 \neq b \Rightarrow b$ because $ 1 \Rightarrow 1 = (1 \to 1) \land (\neg 1 \to \neg 1) = 1 \land 1 = 1$ but $ b \Rightarrow b = (b \to b) \land (\neg b \to \neg b) = b \land b = b$.
	Since both $\mathcal{N}4$ and~$\NS$ are algebraizable logics, this immediately entails that 
	neither $\mathcal{N}4 \leq \NS$ nor $\NS \leq \mathcal{N}4$.
	In logical terms, one can check that the distributivity axiom
	is valid in $\mathcal{N}4$ but not in~$\NS$, whereas the formula
	$
	(\phi \Rightarrow \phi) \Rightarrow (\psi \Rightarrow \psi)
	$
	is valid in~$\NS$ but not in $\mathcal{N}4$.
	\begin{center}
		\begin{minipage}[T]{.45\textwidth}
			\begin{tabular}{l|l l l l}
				$\to $  & $0$  & $n$ & $b$ & $1$ \\  \hline
				$0$ & $1$ & $1$ & $1$ & $1$ \\ 
				$n$ & $1$ & $1$ & $1$ & $1$ \\ 
				$b$ & $0$ & $n$ & $b$ & $1$ \\ 
				$1$ & $0$ & $n$ & $b$ & $1$ \\
			\end{tabular}\qquad \qquad
			\begin{tabular}{l|l}
				$\neg$	 &  \\ \hline
				$0$ & $1$ \\
				$n$ & $n$ \\ 
				$b$ & $b$ \\ 
				$1$ & $1$ \\ 
			\end{tabular}
		\end{minipage}
		\begin{minipage}[T]{.45\textwidth}
			\label{Fig1}
			$$
			\begin{tikzpicture}[scale=.9]
			\node (1) at (0,1) {$1$};
			\node (n) at (-1,0)  {$n$};
			\node (b) at (1,0) {$b$};
			\node (0) at (0,-1) {$0$};
			\draw (1) -- (n) -- (0) -- (b) -- (1);
			\end{tikzpicture}
			$$
			\captionof{figure}{$A_{4}$}
		\end{minipage}%
	\end{center}
\end{proof}

\subsection{$\mathcal{N}3$}
\begin{definition} Nelson's logic $\mathcal{N}3 = \langle \mathbf{Fm}, \vdash_{\mathcal{N}3} \rangle$ is
	the axiomatic extension of $\mathcal{N}4$ obtained by adding the following axiom:
	\begin{description}
		\item[\texttt{(N13)}] $\neg \phi \to (\phi \to \psi)$.
	\end{description}
\end{definition}

\begin{proposition} $\mathcal{N}3$ is a proper extension of~$\NS$.
\end{proposition}
\begin{proof} 
	It is known from \cite{SpVe08a} that every $\mathcal{N}3$-lattice 
	(the algebraic counterpart of $\mathcal{N}3$) 
	satisfies all properties of our Definition~\ref{def:slinha}, and therefore every $\mathcal{N}3$-lattice is an~$\NS$-algebra. On the other hand, the logic $\mathcal{N}3$ was defined as an axiomatic extension of $\mathcal{N}4$, therefore it is distributive too, whereas~$\NS$-algebras need not be distributive (see~\cite[Example 5.1]{Igpl}).
\end{proof}

\section{Future Work} 

We have studied~$\NS$ in two directions, through a proof-theoretic approach and through algebraic methods. Concerning the proof-theoretic approach, we have introduced a finite Hilbert-style calculus for~$\NS$. An interesting question that still remains  is about other types of calculi. In this sense we would find it attractive to be able to present a sequent calculus for~$\NS$ enjoying a cut-elimination theorem, so that it could be used to determine, among other things, whether~$\NS$ is decidable and enjoys the Craig interpolation theorem.

As observed in Theorem~\ref{deduction}, if we let $\phi \to \psi :=  \phi \Rightarrow ( \phi \Rightarrow  \psi)$, then the weak implication $\rightarrow$ enjoys a version of the Deduction Metatheorem;
this suggests that  the connective $\to$ has a special logical role within~$\NS$, whereas $\Rightarrow$ 
is the key operation on the corresponding algebras.
It is well known that the logic $\mathcal{N}3$ as well as
its algebraic counterpart 
can be equivalently axiomatized by taking either the weak or
the strong implication 
as primitive, defining $\to$ from 
$\Rightarrow$ 
as shown above. 
%
The analogous result for $\mathcal{N}4$ has been harder to prove (see \cite{SpinXX}), and 
the corresponding definition is
$\phi \rightarrow \psi := (\phi \land (((\phi \land (\psi \Rightarrow \psi))\Rightarrow \psi) \Rightarrow ((\phi \land (\psi \Rightarrow \psi)) \Rightarrow \psi))) \Rightarrow ((\phi \land (\psi \Rightarrow \psi )) \Rightarrow \psi)$. 

We can ask a similar 
question about the logic~$\NS$ and its algebraic counterpart: namely, 
given that Nelson's  axiomatization as well as ours have $\Rightarrow$ as primitive, is it also possible to axiomatize~$\NS$(-algebras) by taking  the weak implication $\to $ 
as primitive? This question is related to certain algebraic properties that  $\to $ enjoys on $\NS$-algebras. In fact, we have shown 
in~\cite[Theorem 4.5]{Igpl}
that, analogously to $\mathcal{N}3$-lattices, $\NS$-algebras are a variety of
\emph{weak Brouwerian semilattices with filter-preserving operations}~\cite[Definition 2.1]{BKP-EDPC-II},
which means that they possess  an intuitionistic-like internal structure, where
a \emph{weak relative pseudo-complementation} operation (an intuitionistic-like implication)
is given precisely by the weak implication. 
This suggests that one may in fact hope to be able to  view (and axiomatize) $\NS$ as a conservative expansion of 
some intuitionistic-like positive logic by a strong (involutive) negation, 
as has been the case of  $\mathcal{N}3$ and  $\mathcal{N}4$.

As hinted above, a more detailed study of $\NS$-algebras can be found in the companion paper~\cite{Igpl}. Some questions regarding the variety of $\NS$-algebras, its extensions, congruences and more relations between $\NS$-algebras and other well-known algebras are investigated there. Another question that is still open is which 
logic
(class of algebras) is the infimum of $\NS$(-algebras) and $\mathcal{N}4$(-lattices) --- it is easy to see that 
the least logic extending $\NS$ and  $\mathcal{N}4$ is precisely $\mathcal{N}3$.



\bibliographystyle{amsplain}

\section*{Appendix: Proofs of the some of the main results}

\textbf{Theorem \ref{alg}.} The calculus $\vdash_{\NS}$ is implicative, and thus algebraizable.
\begin{proof}
	In the case of~$\NS$, the term~$\alpha(\phi, \psi)$ may be chosen to be $\phi \Rightarrow \psi$.
	We will make below free use of Prop.~\ref{p:eqv}.
	
	\rm{[IL1]} follows immediately from axiom \texttt{(A1)}, while
	\rm{[IL2]} follows from rule $\mathrm{\texttt{(E)}}$.
	\rm{[IL3]}  follows from \texttt{(MP)}
	and 
	\rm{[IL4]} follows from  $\mathrm{(\Rightarrow \texttt{r})}$.
	We are left with proving that $\Rightarrow$ respects \rm{[IL5]} for each connective $ \bullet \in \{ \land, \lor, \Rightarrow, \neg \}$.
	\begin{description}
		\item[$(\neg)$] 
		$\{ (\phi \Leftrightarrow \psi), (\psi \Leftrightarrow \phi) \} \vdash_{\NS}  \neg \phi \Leftrightarrow \neg \psi$ 
		holds by axiom \texttt{(A5)} and the (derived) rule \texttt{(MP)}.
		\medskip
		
		\item[$(\land)$] 
		We must prove that $\{ (\phi_{1} \Leftrightarrow \psi_{1}), (\phi_{2} \Leftrightarrow \psi_{2})\} \vdash_{\NS}  (\phi_{1}\land \phi_{2}) \Leftrightarrow  (\psi_{1} \land \psi_{2})$. From Proposition~\ref{p:ths}.1--2 we have
		$\vdash_{\NS} (\phi_{1} \land \phi_{2}) \Rightarrow \phi_{1}$ and $\vdash_{\NS} (\phi_{1} \land \phi_{2}) \Rightarrow \phi_{2}$.  Then:
		%
		%
		$$\infer[\mathrm{(\land \texttt{r})}]
		{(\phi_{1} \land \phi_{2}) \Rightarrow (\psi_{1} \land \psi_{2}) }{\infer[\texttt{(E)}]{(\phi_{1} \land \phi_{2}) \Rightarrow \psi_{1}}{\infer[]{(\phi_{1} \land \phi_{2}) \Rightarrow \phi_{1}}  \quad \phi_{1} \Rightarrow \psi_{1}} \quad \infer[\texttt{(E)}]{(\phi_{1} \land \phi_{2}) \Rightarrow \psi_{2}}{\infer[]{(\phi_{1} \land \phi_{2}) \Rightarrow \phi_{2}}{}  \quad \phi_{2} \Rightarrow \psi_{2}}} $$
		The remainder of the proof is analogous. \medskip
		
		\item[$(\lor)$] We must prove that $\{ (\phi_{1} \Leftrightarrow \psi_{1}), (\phi_{2} \Leftrightarrow \psi_{2})\} \vdash_{\NS}  (\phi_{1}\lor \phi_{2}) \Leftrightarrow  (\psi_{1} \lor \psi_{2})$. 
		From Proposition~\ref{p:ths}.3--4, $\psi_{1} \Rightarrow (\psi_{1} \lor \psi_{2})$ and $\psi_{2} \Rightarrow (\psi_{1} \lor \psi_{2})$ are derivable.  Then:
		%
		$$\infer[\mathrm{(\lor \texttt{l1})}] 
		{(\phi_{1} \lor \phi_{2}) \Rightarrow (\psi_{1} \lor \psi_{2})}{\infer[\texttt{(E)}]{\phi_{1} \Rightarrow (\psi_{1} \lor \psi_{2})}{\phi_{1} \Rightarrow \psi_{1} \quad \infer[]{\psi_{1} \Rightarrow (\psi_{1} \lor \psi_{2})}{}}  \quad \infer[\texttt{(E)}]{\phi_{2} \Rightarrow (\psi_{1} \lor \psi_{2})}{\phi_{2} \Rightarrow \psi_{2} \quad \infer[]{\psi_{2} \Rightarrow (\psi_{1} \lor \psi_{2})}{}
		}} $$
		The remainder of the proof is analogous. \medskip
		
		\item[$(\Rightarrow)$] We must prove that $\{ (\theta \Leftrightarrow \phi), (\psi \Leftrightarrow \gamma) \}\vdash_{\NS} (\theta \Rightarrow \psi) \Leftrightarrow (\phi \Rightarrow \gamma)$. This time, 
		we have:
		$$\infer[\mathrm{(\Rightarrow \texttt{l})}]
		{\phi \Rightarrow ((\theta \Rightarrow \psi) \Rightarrow \gamma)}{\phi \Rightarrow \theta  \qquad \psi \Rightarrow \gamma}$$
		Taking $\psi$ as $\theta\Rightarrow\psi$ in Prop.~\ref{p:ths}.5, 
		we have: 
		%
		$$
		\infer[\mathrm{\texttt{(MP)}}]{(\theta \Rightarrow \psi) \Rightarrow (\phi \Rightarrow \gamma)}{\phi \Rightarrow ((\theta \Rightarrow \psi) \Rightarrow \gamma) \quad \infer[]{(\phi \Rightarrow ((\theta \Rightarrow \psi) \Rightarrow \gamma)) \Rightarrow ((\theta \Rightarrow \psi) \Rightarrow (\phi \Rightarrow \gamma))}{}}
		$$
		The remainder of the proof is analogous.
	\end{description}
\end{proof}

\noindent \textbf{Proposition \ref{prop:coisas}.} Let  $\mathbf{A}$  be an $\NS$-algebra and let $a,b, c \in A$. Then:
\begin{enumerate}
	\item 
	$a \Rightarrow a = 
	1 = \neg 0$. 
	
	\item The relation $\leq$  defined by setting
	$a \leq b $ iff $a \Rightarrow b = 1 $, 
	is a partial order with maximum~$1$ and
	minimum $0$. 
	
	\item $  a \Rightarrow b = \neg b \Rightarrow \neg a$.
	
	\item $a \Rightarrow (b \Rightarrow c) = b \Rightarrow (a \Rightarrow c)$. 
	
	\item $\neg \neg a = a $ and $a \Rightarrow 0 = \neg a $.
	
	\item $\langle A,*,1\rangle$ is a commutative monoid.
	
	\item $(a * b) \Rightarrow c = a \Rightarrow (b \Rightarrow c)$.
	
	\item 	The pair $( *, \Rightarrow )$ is residuated with respect to $\leq$, i.e., 
	\
	$
	a * b \leq c \  \textrm{ iff } \ b \leq a \Rightarrow c.
	$ 
	
	\item $a^2 \leq a^3$.
	
	\item $\langle A, \land, \lor \rangle$ is a lattice with order $\leq$. 
	
	\item $(a \lor b)^2 \leq a^2 \lor b^2$.
	
\end{enumerate}

\begin{proof}
	
	\begin{enumerate} 
		
		\item This follows from the fact that $\NS$ is an  implicative logic, see \cite[Lemma 2.6]{Font16}.
		In particular, $\neg 0 = 0 \Rightarrow 0 = 1$.
		\smallskip
		\item By $\mathsf{E}(\texttt{A2})$ we have that 0 is the minimum element with respect to the order $\leq$. The rest easily follows from the fact that $\NS$ is implicative.
		\smallskip
		\item This follows from $\mathsf{E}(\texttt{A5})$ and item 2 above.
		\smallskip
		\item By $\mathsf{Q} (\mathrm{\texttt{P}})$ and item 2 above, we have that $d \leq a \Rightarrow (b \Rightarrow c) $ implies $d \leq b \Rightarrow (a \Rightarrow c)$ for all $d \in A$. Then, taking $d = a \Rightarrow (b \Rightarrow c)$, we have $a \Rightarrow (b \Rightarrow c) \leq b \Rightarrow (a \Rightarrow c)$, which easily implies the desired result.
		\smallskip
		\item The identity $\neg \neg a = a$ follows from item 2 above together with $\mathsf{Q}(\mathrm{\neg \neg \texttt{l}})$ and $\mathsf{Q}(\mathrm{\neg \neg \texttt{r})}$. By item 3 above, $a \Rightarrow 0 =  \neg 0 \Rightarrow \neg a = 1 \Rightarrow \neg a = \neg a$. The last identity holds good because, on the one hand, by $\mathsf{Q} \mathrm{(\Rightarrow \texttt{l})}$ we have that $1 \leq 1$ and $\neg a \leq \neg a$ implies $1 \Rightarrow \neg a \leq \neg a$. On the other hand, by item 1 we have $\neg a \Rightarrow \neg a \leq 1$ and so we can apply $\mathsf{Q} \mathrm{(\Rightarrow \texttt{r})}$ to obtain $1 \Rightarrow (\neg a \Rightarrow \neg a) = 1$. By item 4, we have $1 \Rightarrow (\neg a \Rightarrow \neg a) = \neg a \Rightarrow (1 \Rightarrow \neg a)$, hence we conclude that $\neg a \Rightarrow (1 \Rightarrow \neg a) = 1$ and so, by item 2, $\neg a \leq 1 \Rightarrow \neg a$.
		\smallskip
		\item As to commutativity, using items 3 and 5 above, we have $a * b = \neg(a \Rightarrow \neg b) = \neg ( \neg \neg b \Rightarrow \neg a) = \neg ( b \Rightarrow \neg a) = b * a$.  
		As to associativity, 
		using 3, 5, $		\mathsf{Q}(\mathrm{\neg \neg \texttt{r})}$ and $		\mathsf{Q}(\mathrm{\neg \neg \texttt{l})}$,
		we have $(a * b) * c = 
		\neg(\neg(a \Rightarrow \neg b) \Rightarrow \neg c) =
		\neg( \neg \neg c \Rightarrow \neg \neg (a \Rightarrow \neg b) =
		\neg( c \Rightarrow (a \Rightarrow \neg b))
		=  \neg( a \Rightarrow (c \Rightarrow \neg b)) = \neg( a \Rightarrow (b \Rightarrow \neg c)) = \neg( a \Rightarrow \neg \neg (b \Rightarrow \neg c)) = a * (b * c) 
		$.
		As to 1 being the neutral element, using items 1 and 5 above, we have $a * 1 = a * \neg 0 =\neg (a \Rightarrow \neg \neg 0 )=  \neg (a \Rightarrow 0 ) = \neg \neg a = a $.
		\smallskip
		\item Using items 2, 3, 5 and 6 above, we have  
		$(a * b) \Rightarrow c = \neg (a \Rightarrow \neg b) \Rightarrow c 
		=  \neg c \Rightarrow \neg \neg (a \Rightarrow \neg b) 
		= \neg c \Rightarrow (a \Rightarrow \neg b) = a \Rightarrow (\neg c \Rightarrow \neg b)
		= a \Rightarrow ( \neg \neg b \Rightarrow \neg \neg c) 
		= a \Rightarrow ( b \Rightarrow c)$.
		\smallskip
		\item 
		By item 2 above, we have
		$a * b \leq c$ iff $(a * b) \Rightarrow c = 1$ iff, by item 7,
		$a \Rightarrow (b \Rightarrow c) = 1$ iff, by item 6,
		$b \Rightarrow ( a \Rightarrow c) = 1$ iff, by 2 again, $b \leq a \Rightarrow c$.
		\smallskip
		\item 
		By  $		\mathsf{Q}(\texttt{C})$ we have that $a^3 \leq c$ implies $a^2 \leq c$ for all $c \in A$. Then, taking $c = a^3$, we have $a^2 \leq a^3$.
		\smallskip
		\item 
		We check that $a \land b$ is the infimum of the set $\{a,b \}$ with respect to $\leq$. First of all, we have $a \land b \leq a$ and $a \land b \leq b$ by $\mathsf{Q} \mathrm{(\land \texttt{l1})}$, $\mathsf{Q} \mathrm{(\land \texttt{l2})}$ and item 2 above. Then, assuming $c \leq a$ and $c \leq b$, we have $c \leq a \land b$ by $\mathsf{Q} \mathrm{(\land \texttt{r})}$. An analogous reasoning, using $\mathsf{Q} \mathrm{(\lor \texttt{r1})}$, $\mathsf{Q} \mathrm{(\lor \texttt{r2})}$  and  $\mathsf{Q} \mathrm{(\lor \texttt{l1})}$ shows that $a \lor b$ is the supremum of $\{a,b \}$.
		\smallskip
		\item 
		By item 10 we have that $a^2   \leq  a^2 \lor b^2$ and $b^2   \leq  a^2 \lor b^2$. Hence, by item 8, we have $a  \leq  a \Rightarrow (a^2 \lor b^2)$ and $b  \leq  b \Rightarrow (a^2 \lor b^2)$. By item 2 we have then $a  \Rightarrow ( a \Rightarrow (a^2 \lor b^2)) = b  \Rightarrow ( b \Rightarrow (a^2 \lor b^2)) = 1$, hence we can use $\mathsf{Q} \mathrm{(\lor \texttt{l2})}$ to obtain $(a \lor b)  \Rightarrow ( (a \lor b) \Rightarrow (a^2 \lor b^2)) = 1$. Then  items 2 and 8 give us $(a \lor b)^2 \leq a^2 \lor b^2$, as was to be proved.
	\end{enumerate} 
\end{proof}

\noindent \textbf{Lemma~\ref{lemma1}.}
\begin{enumerate}
	\item	Any CIBRL 
	satisfies the equation $ (x\lor y)*z \approx (x*z) \lor (y*z)$.
	\item Any 3-potent CIBRL 
	satisfies 
	$x^{2} \lor y^{2} \approx (x^{2} \lor y^{2})^{2}$.
	\item  Any 3-potent CIBRL 
	satisfies 
	$(x\lor y^{2})^{2} \approx (x\lor y)^{2}.$
	\item Any 3-potent CIBRL 
	satisfies  $( x\lor y)^{2} \approx x^{2} \lor y^{2}$.
\end{enumerate}
\begin{proof}
	\begin{enumerate}
		\item  See~\cite[Lemma 2.6]{GaJiKoOn07}.
		\smallskip
		\item Let $a,b$ be arbitrary elements of a given 3-potent CIBRL. From $a^{2} \leq (a^{2} \lor b^{2})$ and $b^{2} \leq (a^{2} \lor b^{2})$, using monotonicity of~$*$, we have $a^{4} \leq (a^{2} \lor b^{2})^{2}$ and  $b^{4} \leq (a^{2} \lor b^{2})^{2}$. Using 3-potency, 
		the latter inequalities simplify to	$a^{2} \leq (a^{2} \lor b^{2})^{2}$ and $b^{2} \leq (a^{2} \lor b^{2})^{2}$. Thus, $a^{2} \lor b^{2} \leq (a^{2} \lor b^{2})^{2}$.
		\smallskip
		\item We have $a\lor b^{2} \leq a\lor b$ from monotonicity of~$*$ and supremum of~$\lor$, therefore $(a\lor b^{2})^{2} \leq (a\lor b)^{2}$. For the converse, we have that $a*b \leq a$, whence  $a*b \leq a\lor b^{2}$. Also $a^2 \leq a\lor b^{2}$ and $b^{2} \leq a\lor b^{2}$. By supremum of~$\lor$, $a^{2} \lor (a*b) \lor b^{2} \leq a\lor b^{2}$. But $a^{2} \lor (a*b) \lor b^{2} = (a\lor b)^{2}$ by Lemma~\ref{lemma1}.1, so $(a\lor b)^{2} \leq  a\lor b^{2}$. Using the monotonicity of~$*$, $(a\lor b)^{4} \leq  (a\lor b^{2})^{2}$ and from 3-potency we have $(a\lor b)^{2} \leq  (a\lor b^{2})^{2}$.
		\smallskip
		\item From Lemma~\ref{lemma1}.2 we have $a^{2} \lor b^{2} = (a^{2} \lor b^{2})^{2}$, 
		and from Lemma~\ref{lemma1}.3 we have
		$(a^{2} \lor b^{2})^{2} = (a^{2} \lor b)^{2} = (b \lor a^{2})^{2} = (b \lor a)^{2}$.
	\end{enumerate}
\end{proof}

\noindent \textbf{Proposition \ref{equivalence}.}
Let $\mathbf{A} = \langle A, \land, \lor, *, \Rightarrow, 0, 1 \rangle$ be an $\NS^{\prime}$-algebra. Defining $\neg x := x \Rightarrow 0$, 
we have that  $\mathbf{A'} = \langle A, \land, \lor, \Rightarrow,  \neg, 0, 1 \rangle$ is an $\NS$-algebra.

\begin{proof} 
	Let $\mathbf{A}$ be an $\NS^{\prime}$-algebra. We first consider the equations corresponding to the axioms of $\NS$. As  $a \leq b $ iff $a \Rightarrow b = 1 $, we will write the former rather than the latter. \smallskip\\
	\noindent
	\textit{Equations} \\
	The equation
	$\mathsf{E} ( \texttt{A1} )$ easily follows from integrality. We have
	$\mathsf{E} ( \texttt{A2} )$  from the fact that $0$ is the minimum element of $\mathbf{A}$.
	From the definition of~$\neg$ in~$\NS^{\prime}$ and from $\mathsf{E} ( \texttt{A1} )$ we see that $\mathsf{E} ( \texttt{A3} )$ holds.
	We know that $1 := \neg 0$, therefore we have $\mathsf{E} ( \texttt{A4} )$. As  $\mathbf{A}$  is  involutive, it follows that $\mathsf{E} ( \texttt{A5} )$ holds. We are still to prove the equation  $\mathsf{E}(\Delta(\varphi, \varphi))$. For that, see that we need to prove the identity  $(\phi \Rightarrow \phi) \land (\phi \Rightarrow \phi) = 1$, and we already know that $\phi \Rightarrow \phi = 1$, therefore also $(\phi \Rightarrow \phi) \land (\phi \Rightarrow \phi) = 1$.
	\medskip\\
	\textit{quasiequations}
	\\
	$\mathsf{Q}(\texttt{P})$ follows from the commutativity of~$*$ and from the identity $(a *b ) \Rightarrow c = a \Rightarrow (b \Rightarrow c)$.
	$\mathsf{Q}(\texttt{C})$ follows from 3-potency: since  $ a^2 \leq a^3$, we have that
	$a^3 \Rightarrow b = 1$ implies $a^2 \Rightarrow b=1$.\smallskip \\
	$\mathsf{Q}(\texttt{E})$ follows from the fact that $\mathbf{A}$ has a partial order $\leq$ that is determined by the implication $\Rightarrow$.
	To prove $\mathsf{Q}(\mathrm{\Rightarrow \texttt{l}})$, suppose $a \leq b$ and $c \leq d$. From $c \leq d$, as $b \Rightarrow c \leq  b \Rightarrow c$, using residuation we have that $b*(b \Rightarrow c) \leq c \leq d$, therefore $b*(b \Rightarrow c) \leq d$ and therefore $b \Rightarrow c \leq b \Rightarrow d$. Note that as $a \leq b$, using residuation we have that $a*(b \Rightarrow d) \leq b*(b \Rightarrow d) \leq d$, therefore $b \Rightarrow d \leq a \Rightarrow d$ and then $b \Rightarrow c \leq a \Rightarrow d$. Now, since $b \Rightarrow c \leq a \Rightarrow d$ iff $a*(b \Rightarrow c) \leq d$ iff $a \leq (b \Rightarrow c) \Rightarrow d$, we obtain thus the desired result.\smallskip \\
	For $\mathsf{Q}(\mathrm{\Rightarrow \texttt{r}})$  we need to prove that if $d = 1$, then  $b \Rightarrow d = 1$. 
	This follows immediately from integrality.\smallskip \\
	quasiequations $\mathsf{Q}(\mathrm{\land \texttt{l1}})$, $\mathsf{Q}(\mathrm{\land \texttt{l2}})$, $\mathsf{Q}(\mathrm{\land \texttt{r}})$, $\mathsf{\mathrm{Q}}(\mathrm{\lor \texttt{l1}})$, $\mathsf{Q}(\mathrm{\lor \texttt{r1}})$ and $\mathsf{Q}(\mathrm{\lor \texttt{r2}})$ follow
	straightforwardly from the fact that $\mathbf{A}$ is partially ordered and the order is determined by the implication.\smallskip \\
	In order to prove $\mathsf{\mathrm{Q}}(\mathrm{\lor \texttt{l2}})$, notice that $(b \lor c)^2 \leq b^2 \lor c^2$ by 
	Lemma~\ref{lemma1}.4.
	Suppose $b^2 \leq d$ and $c^2 \leq d$, then since $\mathbf{A}$ is a lattice, we have $b^2 \lor c^2 \leq d$ and as $(b \lor c)^2 \leq b^2 \lor c^2$ we conclude that $(b \lor c)^2 \leq d$ and thus $(b \lor c)^2 \Rightarrow d = 1$.\smallskip \\
	As to $\mathsf{Q}(\mathrm{\neg \Rightarrow \texttt{l}})$, by integrality we have $b * c \leq b$ and $b * c \leq c$. Thus $b * c \leq b \land c$. Now, if $b \land c \leq d$, then $b * c \leq d$.\smallskip \\ 
	In order to prove $\mathsf{Q}(\mathrm{\neg \Rightarrow \texttt{r}})$, suppose $d^2 \leq b \land c$. Using monotonicity of~$*$, we have $d^2*d^2 \leq (b \land c)*(b \land c)$, i.e., $d^4 \leq (b \land c)^2$. Using 3-potency, we have $d^4 = d^2$, therefore $d^2 \leq (b \land c)^2$.   Since $(b \land c)^2 \leq b * c$, we have $d^2 \leq (b \land c)^2 \leq b* c$, i.e., $d^2 \leq b* c$.\smallskip \\
	$\mathsf{Q}(\mathrm{\neg \land \texttt{l}})$, $\mathsf{Q}(\mathrm{\neg \land \texttt{r}})$, $\mathsf{Q}(\mathrm{\neg \lor \texttt{l}})$  and $\mathsf{Q}(\mathrm{\neg \lor \texttt{l}})$ follow from the De Morgan's Laws (cf.~\cite[Lemma 3.17]{GaJiKoOn07}).\smallskip \\
	Finally,
	we have $\mathsf{Q}(\mathrm{\neg \neg \texttt{l}})$ and $\mathsf{Q}(\mathrm{\neg \neg \texttt{r}})$ from $\mathbf{A}$ being involutive.\smallskip \\
	It remains to be proven that the quasiequation $ \mathsf{E}(\Delta(\varphi,\psi)) $ implies $  \varphi \approx \psi$, that is, if $\varphi \Rightarrow \psi = 1$ and $\psi \Rightarrow \varphi = 1$, then $\varphi = \psi$. As 1 is the maximum of the algebra, we have that $\varphi \Rightarrow \psi = 1$ and $\psi \Rightarrow \varphi = 1$, therefore $\varphi \leq \psi$ and $\psi \leq \varphi$.  As~$\leq$ is an order relation, it follows that  $\varphi = \psi$.
\end{proof}

\end{document}